\documentclass[11pt]{amsart}

\usepackage[usenames,dvipsnames]{pstricks}
\usepackage[mathscr]{eucal}

\newcommand{\numberseries}{\bfseries}   

\newlength{\thmtopspace}                
\newlength{\thmbotspace}                
\newlength{\thmheadspace}               
\newlength{\thmindent}                  

\newtheoremstyle{fixed bf head,slanted body}
                {\thmtopspace}{\thmbotspace}{\slshape}
                {\thmindent}{\bfseries}{.}{\thmheadspace}
                {{\numberseries \thmnumber{#2\;}}\thmname{#1}\thmnote{ (#3)}}

\newtheoremstyle{variable bf head,slanted body}
                {\thmtopspace}{\thmbotspace}{\slshape}
                {\thmindent}{\bfseries}{.}{\thmheadspace}
                {{\numberseries \thmnumber{#2\;}}\thmname{#1}\thmnote{ #3}}

\newtheoremstyle{fixed bf head,upright body}
                {\thmtopspace}{\thmbotspace}{\upshape}
                {\thmindent}{\bfseries}{.}{\thmheadspace}
                {{\numberseries \thmnumber{#2\;}}\thmname{#1}\thmnote{ (#3)}}

\newtheoremstyle{numbered paragraph}
                {\thmtopspace}{\thmbotspace}{\upshape}
                {\thmindent}{\upshape}{}{\thmheadspace}
                {{\numberseries \thmnumber{#2.}}}

\usepackage[all,2cell,ps]{xy}\usepackage{amssymb,amstext,amsmath,amscd,amsthm,amsfonts,enumerate,graphicx,latexsym,color}
\usepackage{mathptmx}

\usepackage[all]{xy} \SelectTips{eu}{}
\usepackage{latexsym,amsfonts,amssymb}
\usepackage{hyperref}
\usepackage{mathptmx}
\usepackage{nicefrac}
\usepackage{array}
\usepackage{todonotes}

\theoremstyle{definition}
\newtheorem{thm}{Theorem}[section]
\newtheorem{chunk}[thm]{\hspace*{-1.065ex}\bf}

\newtheorem{rmk}[thm]{Remark}
\newtheorem{prop}[thm]{Proposition}
\newtheorem{cor}[thm]{Corollary}
\newtheorem{ques}[thm]{Question}
\newtheorem{prob}[thm]{Problem}

\theoremstyle{definition}
\newtheorem{dfn}[thm]{Definition}
\newtheorem{eg}[thm]{Example}

\theoremstyle{remark}

\def\p{\mathfrak{p}}

\newcommand{\rpm}{\raisebox{.2ex}{$\scriptstyle\pm$}}

\newcommand{\CC}{\mathbb{C}}

\newcommand{\ZZ}{\mathbb{Z}}
\newcommand{\NN}{\mathbb{N}}

 \DeclareMathOperator{\coker}{coker}

\DeclareMathOperator{\Ho}{\operatorname{\mathbb{I}}}
\DeclareMathOperator{\Hdim}{\operatorname{\mathbf{\mathbb{I}}}}
\DeclareMathOperator{\rHdim}{\operatorname{\mathsf{red-\mathbb{I}}}}
\DeclareMathOperator{\rpd}{\operatorname{\mathsf{red-pd}}}

\DeclareMathOperator{\rGdim}{\operatorname{\mathsf{red-G-dim}}}
\DeclareMathOperator{\rCI}{\operatorname{\mathsf{red-CI-dim}}}

\DeclareMathOperator{\pd}{\operatorname{\mathsf{pd}}}
\DeclareMathOperator{\CI}{\operatorname{\mathsf{CI-dim}}}
\DeclareMathOperator{\G}{\operatorname{\mathsf{G-dim}}}

\DeclareMathOperator{\PPP}{\operatorname{P}}

\DeclareMathOperator{\Ext}{\operatorname{\mathsf{Ext}}}

\DeclareMathOperator{\depth}{\operatorname{\mathsf{depth}}}

\DeclareMathOperator{\Hom}{\operatorname{\mathsf{Hom}}}

\def\m{\mathfrak{m}}

\def\urltilda{\kern -.15em\lower .7ex\hbox{\~{}}\kern .04em}\def\urldot{\kern -.10em.\kern -.10em}\def\urlhttp{http\kern -.10em\lower -.1ex\hbox{:}\kern -.12em\lower 0ex\hbox{/}\kern -.18em\lower 0ex\hbox{/}}

\numberwithin{equation}{section}

\begin{document}
\baselineskip=15pt
\title{Reducing invariants and total reflexivity}			
\author{Tokuji Araya}
\address{Department of Applied Science, Faculty of Science, Okayama University of Science, Ridaicho, Kitaku, Okayama 700-0005, Japan.}
\email{araya@das.ous.ac.jp}
\author{Olgur Celikbas}
\address{Department of Mathematics, West Virginia University, 
Morgantown, WV 26506 U.S.A}
\email{olgur.celikbas@math.wvu.edu}


\begin{abstract} Motivated by a recent result of Yoshino, and the work of Bergh on reducible complexity, we introduce reducing versions of invariants of finitely generated modules over commutative Noetherian local rings. Our main result considers modules which have finite reducing Gorenstein dimension, and determines a criterion for such modules to be totally reflexive in terms of the vanishing of Ext. Along the way we give examples and applications, and in particular, prove that a Cohen-Macaulay local ring with canonical module is Gorenstein if and only if the canonical module has finite reducing Gorenstein dimension.
\end{abstract}
\maketitle

\section{Introduction}
Throughout $R$ denotes a commutative Noetherian local ring with unique maximal ideal $\m$ and residue field $k$. Moreover, each $R$-module is assumed to be finitely generated. For standard, unexplained basic terminology and notations, we refer the reader to \cite{AuBr}, \cite{Av2}, \cite{BH} and \cite{Vla}.

An $R$-module $M$ is called \emph{totally reflexive} if $\Ext^i_R(M,R) = 0 = \Ext^i_R(M^{\ast},R)$ for all $i\geq 1$, and $M$ is reflexive, i.e., $M \cong M^{\ast\ast}$, where  $M^{\ast}=\Hom_R(M ,R)$. This definition is valid over each Noetherian ring (not necessarily local) and is due to Auslander and Bridger \cite{AuBr}. The \emph{Gorenstein dimension} $\G_R(M)$ of an $R$-module $M$ is defined in terms of the length of a resolution of totally reflexive modules, and has been studied extensively in the literature. Note that the totally reflexive modules are precisely the nonzero modules of Gorenstein dimension zero. In 2006 Jorgensen and {\c{S}}ega \cite{JS} proved that the conditions defining total reflexivity are independent of each other: one of the examples they constructed is a local Artinian ring $R$, and 
an $R$-module $M$ such that $\Ext^i_R(M,R) = 0 \neq \Ext^i_R(M^{\ast},R)$ for all $i\geq 1$. The work of Jorgensen and {\c{S}}ega also yields a non-reflexive module over $R$ with the same vanishing conditions. Therefore, it seems natural to us to consider the following problem:

\begin{prob} \label{zor} Let $M$ be an $R$-module. Determine conditions on $R$, or on $M$, so that the vanishing of $\Ext^i_R(M,R)$ for all $i\geq 1$ forces $M$ to be totally reflexive.
 \pushQED{\qed} 
\qedhere
\popQED	
\end{prob}

Let us note, in general, we even do not know if the vanishing of $\Ext^i_R(M,R)$ for all $i\geq 1$ forces $M$ to be Cohen-Macaulay. Recently Yoshino \cite{Yuj} extended the stable module theory of Auslander and Bridger \cite{AuBr} to the \emph{stable complex theory} and, by using his theory, proved that, if $R$ is generically Gorenstein and $X$ is an arbitrary complex of finitely generated projective $R$-modules, then $X$ is exact if and only if $X^{\ast}$ is exact. Yoshino's work yielded the following beautiful, far-reaching theorem concerning Problem \ref{zor}.

\begin{thm} (Yoshino \cite{Yuj}) \label{Yos} Let $M$ be an $R$-module. If $R$ is generically Gorenstein (i.e., $R_{\p}$ is Gorenstein for each associated prime ideal $\p$ of $R$), then one has $\G(M)=\sup\{i \in \ZZ: \Ext^{i}_R(M,R)=0\}$.  \pushQED{\qed} 
\qedhere
\popQED	
\end{thm}

The aim of this paper is to consider Problem \ref{zor} and obtain a result in the direction of Theorem \ref{Yos}. Our main result is motivated by, besides Theorem \ref{Yos}, the \emph{reducible complexity} definition of Bergh \cite{Be}: we introduce \emph{reducing} versions of homological invariants (in particular those of homological dimensions) of $R$-modules, and prove the following:

\begin{thm} \label{mainthmintro} Let $M$ be an $R$-module which has finite \emph{reducing Gorenstein dimension}. Then one has $\G(M)=\sup\{i \in \ZZ: \Ext^{i}_R(M,R)=0\}$.  \pushQED{\qed} 
\qedhere
\popQED	
\end{thm}

In general, a module can have finite reducing Gorenstein dimension, even if it has infinite Gorenstein dimension: we give and discuss such examples, as well as the definition of reducing homological dimensions, in Section 2. We prove Theorem \ref{mainthmintro} in section 4, but defer the proofs of several preliminary results to Section 5. Moreover, in Section 3, we give an application on testing the Gorenstein property in terms of the reducing Gorenstein dimension and prove the following; see Theorem \ref{t2} and Corollary \ref{面白い}.

\begin{prop} \label{propintro} Let $M$ be an $R$-module which has finite reducing Gorenstein dimension. Then:
\begin{equation}\notag{}
\Ext^{i}_R(M,M)=0  \text{ for all } i\gg 0 \Longrightarrow \Ext^{i}_R(M,R)=0  \text{ for all } i\gg 0 \Longleftrightarrow \G_R(M)<\infty.  
\end{equation}
In particular, if $R$ is Cohen-Macaulay with canonical module $\omega$, then $R$ is Gorenstein if and only if $\omega$ has finite reducing Gorenstein dimension if and only if $\omega$ has finite Gorenstein dimension.
\pushQED{\qed} 
\qedhere
\popQED	
\end{prop}

The conclusion of Proposition \ref{propintro} is already known if $M$ has reducible complexity: in fact, in this case, $M$ would have finite projective dimension; see \cite[3.2]{Be}. Let us mention here that, if a module has reducible complexity, then it has finite reducing projective, and hence finite reducing Gorenstein dimension; see Definition \ref{maindfn} and \cite[2.1]{Be}. However, it is easy to find examples of modules that do not have finite complexity (and hence do not have reducible complexity), but have finite reducing projective dimension: in Example \ref{ex1}, $R$ is not a complete intersection so that the complexity of $k$ is not finite.

\section{Definitions and examples}

In the following, $\Hdim$ denotes a \emph{homological invariant} of $R$-modules, i.e., $\Hdim$ denotes a map from the set of isomorphism classes of $R$-modules to the set $\ZZ \cup \{ \rpm \infty \}$.
Classical and well-known examples of such an invariant are homological dimensions including the projective dimension $\Hdim=\pd$ \cite{BH}, Gorenstein dimension $\Hdim=\G$ \cite{AuBr}, and complete intersection dimension $\Hdim=\CI$ \cite{AGP}. Motivated by the \emph{reducible complexity} definition of Bergh \cite{Be}, we define:

\begin{dfn} \label{maindfn} Let $M$ be an $R$-module, and let $\Hdim$ be a \emph{homological invariant} of $R$-modules.

We write $\rHdim(M)<\infty$ provided that there exists a sequence of $R$-modules $K_0,  \ldots, K_r$, positive integers $a_1, \dots, a_r,b_1, \dots, b_r ,n_1, \dots, n_r$, and short exact sequences of the form $0 \to K_{i-1}^{{\oplus a_i}} \to K_{i} \to \Omega^{n_i}K_{i-1}^{{\oplus b_i}} \to 0$ for each $i=1, \ldots r$, where $K_0=M$ and $\Hdim(K_r)<\infty$. If such a sequence of modules exists, then we call $\{K_0,  \ldots, K_r\}$ a \emph{reducing $\Hdim$-sequence} of $M$. 

The \emph{reducible invariant} $\Hdim$ of $M$ is defined as follows:
\begin{equation}\notag{}
\rHdim(M)=\inf\{ r\in \NN \cup \{0\}: \text{there is a reducing $\Hdim$-sequence }  K_0,  \ldots, K_r \text{ of }  M\}.
\end{equation}
We set, $\rHdim(M)=0$ if and only if $\Hdim(M)<\infty$. \pushQED{\qed} 
\qedhere
\popQED	
\end{dfn}

In this paper we will focus on reducing homological dimensions, especially on reducing Gorenstein dimension. If $M$ is an $R$-module that has reducible complexity (e.g., if $\CI_R(M)<\infty$), then $M$ has finite reducing projective dimension. In particular, if $R$ is a complete intersection, then each $R$-module has finite reducible projective dimension; see \cite[2.2]{Be}. We suspect that the converse of this fact is also true. Hence it seems reasonable to ask:

\begin{ques} If each $R$-module has finite reducing projective dimension, then must $R$ be a complete intersection ring? What if each $R$-module has reducing projective dimension at most one?
\pushQED{\qed} 
\qedhere
\popQED	
\end{ques}

The reducing homological invariant of an $R$-module $M$ can be finite, even if the corresponding homological invariant is infinite, i.e., in general, reducing homological dimensions are finer invariants than their corresponding homological dimensions. Next we give several examples and remarks to highlight this point. The following is an example of a module that has infinite Gorenstein, but has finite reducing Gorenstein dimension: 

\begin{eg} \label{ex1} Let $R=\CC[\![x,y]\!]/(x,y)^2$. Then we have that $\G(k)=\pd(k)=\infty$, and the minimal free resolution of $k$ is given by:
$$\xymatrix@C=1.4cm{ 
\cdots \ar[r] & R^{\oplus 4} \ar[r]^{\left(\begin{smallmatrix}
x & y & 0 & 0 \\
0 & 0 & x & y
\end{smallmatrix}\right)} & R^{\oplus 2} \ar[r]^{\left(\begin{smallmatrix}
x & y 
\end{smallmatrix}\right)} & R \ar[r] & k \ar[r] & 0.}$$
Since $\Omega^2k \cong k^{\oplus 4}$, the following short sequence is exact: $0 \to k^{\oplus 4} \to R^{\oplus 2} \to \Omega k \to 0$. 
This yields that $\{ k, R^{\oplus 2} \}$ is a reducing $\pd$-sequence (and so reducing $\G$-sequence) of $k$. Therefore, we conclude that $\rpd(k) =\rGdim(k) =1<\infty$.\pushQED{\qed} 
\qedhere
\popQED	
\end{eg}

The next remark and Proposition \ref{prop1} establish a generalization of Example \ref{ex1}.

\begin{rmk}\label{rmk2} Assume $R$ is not Gorenstein and $\m^2=0$. It is easy to see that there is a short exact sequence $0 \to k^{\oplus e^2} \to R^{\oplus e} \to \Omega k \to 0$, where $e$ is an embedding dimension of $R$. This shows that $\{ k, R^{\oplus e} \}$ is a reducing $\pd$-sequence (and hence reducing $\G$) sequence of $k$. Consequently, we see $\rpd(k)=1=\rGdim(k) $, while $\pd(k) =\infty=\G(k)$. \pushQED{\qed} 
\qedhere
\popQED	
\end{rmk}

\begin{prop}\label{prop1}
The following are equivalent, if $R$ is not Gorenstein, $\m^2=0$ and $M$ is an $R$-module. 
\begin{enumerate}[\rm(i)]
\item $M \cong R^{\oplus \alpha}\oplus k^{\oplus \beta}$ for some $\alpha, \beta\geq 0$.
\item $\rpd(M)<\infty$.
\item $\rpd(M) \leq 1$.
\item $\rCI(M)<\infty$.
\item $\rCI(M) \leq 1$.
\item $\rGdim(M) < \infty$.
\item $\rGdim(M) \leq 1$.
\end{enumerate}
\end{prop}

\begin{proof} Note that, by definitions, it suffices to prove (i) implies (iii), and (vi) implies (i).

Assume (i) holds. Then there is an exact sequence of the form $0 \to k^{\oplus e^2} \to R^{\oplus e} \to \Omega k \to 0$; see Remark \ref{rmk2}. This induces the following commutative diagram with exact rows:
$$\xymatrix{
0 \ar[r] & R^{\oplus e^2\alpha} \oplus k^{\oplus e^2\beta} \ar[r] \ar@{=}[d] & R^{\oplus e^2\alpha} \oplus R^{\oplus e\beta} \ar[r] \ar@{=}[d] & \Omega k^{\oplus \beta} \ar[r] \ar@{=}[d] & 0\\
0 \ar[r] & (R^{\oplus \alpha} \oplus k^{\oplus \beta})^{\oplus e^2} \ar[r] & R^{\oplus (e^2\alpha + e\beta)} \ar[r] & \Omega (R^{\oplus \alpha} \oplus k^{\oplus \beta}) \ar[r]& 0.
}$$
Therefore $\{ R^{\oplus \alpha} \oplus k^{\oplus \beta}, R^{\oplus (e^2\alpha + e\beta)}\}$ is a reducing $\pd$-sequence of $M \cong R^{\oplus \alpha} \oplus k^{\oplus \beta}$. Hence we see that $\rpd (M) \leq 1$. This establishes (iii).

Next we will show that (vi) implies (i). Assume $\rGdim(M) < \infty$. Then there exists a reducing $\G$ sequence, say $\{ K_0,K_1,\dots,K_r\}$, of $M$.
Note that, by definition, there are positive integers $a_1,a_2,\dots,a_r$ and injective maps $K_{i-1}^{\oplus a_i} \to K_i$ for each $i=1, \ldots r$. Hence, setting $K_0=M$, we obtain an injective map $\varphi:M \to K_r$. On the other hand, since $\G(K_r)<\infty$, we conclude that $K_r$ is a free $R$-module; see, for example, \cite[2.4]{YYGdim}. Therefore, $M \cong \Omega(\coker \varphi) \oplus R^{\oplus \alpha}$ for some $\alpha\geq 0$. Since 
$\Omega(\coker \varphi)$ is a $k$-vector space, we have that $M \cong R^{\oplus \alpha} \oplus k^{\oplus \beta}$ for some $\alpha, \beta \geq 0$.
\end{proof}

The ring in Example \ref{ex1} is zero-dimensional. We now proceed to give a higher dimensional example of a ring over which Gorenstein and reducing Gorenstein dimensions are different. Recall that $R$ is said to be \emph{G-regular} \cite{TakG} provided that there are no non-free totally reflexive $R$-modules.

\begin{rmk} \label{rmkEx2} Assume $R$ is G-regular and Cohen-Macaulay with canonical module $\omega$. Assume further there exists a nonzero $R$-module $M$ such that $R$, $\omega$, and $M$ are the only, up to isomorphism, pairwise non-isomorphic, indecomposable maximal Cohen-Macaulay $R$-modules. Then $\Omega M$ is a maximal Cohen-Macaulay $R$-module that has no free summand; see, for example, \cite[9.14(i)]{BLW}. 

Now suppose $\Omega M \cong M^{\oplus r} \oplus \omega^{\oplus s}$ for some $r, s \geq 0$. Pick a maximal $R$-regular sequence $\underline{x}$ (this is empty set if $\depth(R)=0$). Without loss of generality, we may assume $\underline{x}=\{x\}$. Then there is an exact sequence $0 \to \Omega M \to G \to M \to 0$ for some free $R$-module $G$ so that $0 \to \Omega M/x \Omega M \to G/xG$ is exact. Notice $\Omega M/x \Omega M \cong (M/xM)^{\oplus r} \oplus (\omega/x\omega)^{\oplus s}$.  Hence, if $s\geq 1$, then there is an injection $\omega/x\omega \to G/xG$ over the Artinian ring $R/xR$. This implies $\omega/x\omega$ is free over $R/xR$, and so $\omega \cong R$. Thus $s=0$, i.e., $\omega$ cannot be a direct summand of $\Omega M$. Therefore, $r\geq 1$ and $\Omega M \cong M^{\oplus r}$. This yields a short exact sequence of the form $0 \to M^{\oplus r^2} \to F \to \Omega M \to 0$, where $F$ is a free $R$-module. As $M$ is not free, we have $\G(M)=\infty$ so that $\rGdim_R(M)=1<\infty=\G_R(M)$; see Definition \ref{maindfn}. \pushQED{\qed} 
\qedhere
\popQED	
\end{rmk}

\begin{eg} \label{withoutrmkEx2} Let $S=\CC[\![x, y]\!]$ be the formal power series ring, and let $R=\CC[\![x^3, x^2y, xy^2, y^3 ]\!]$ be the $3$rd \emph{Veronese subring} of $S$. Then $S=R \oplus M \oplus \omega$, where $\omega=(x,y)$ is the canonical module of $R$, and $M=(x^2,xy, y^2)$. The set of all  indecomposable maximal Cohen-Macaulay $R$-modules equals the set of all indecomposable $R$-direct summands of $S$, which is $\{R, \omega, M\}$; see \cite[10.5]{Yo}. As $R$ is G-regular, we conclude by Remark \ref{rmkEx2} that $\rGdim_R(M)=1<\infty=\G_R(M)$; see \cite[6.3.6]{BLW} and \cite[5.1]{TakG}. In fact, one can check that $\Omega \omega \cong M$, $\Omega M \cong M^{\oplus 2}$, and there are exact sequences $0 \to \Omega M  \to R^{\oplus 3} \to M \to 0$ and $0 \to M^{\oplus 4} \to F \to \Omega M \to 0$, where $F$ is free.  \pushQED{\qed} 
\qedhere
\popQED	
\end{eg}

\section{An application on testing the Gorenstein property}

It is known that, if $R$ is Cohen-Macaulay with canonical module $\omega$, then $R$ is Gorenstein if and only if $\G_R(\omega)<\infty$; see \cite[1.2]{JLG}. Prior to giving a proof of Theorem  \ref{mainthmintro}, to faciliate further discussion, we give an application of reducing Gorenstein dimension and extend the aforementioned fact about canonical modules. More precisely, we will prove that, if $C$ is a semidualizing $R$-module (see, for example, \cite{SSC}), then $\rGdim(C)<\infty$ if and only if $\G_R(C)<\infty$, i.e., $C$ is totally reflexive. This shows that, if $R$ is Cohen-Macaulay with canonical module $\omega$, then $R$ is Gorenstein if and only if $\rGdim(\omega)<\infty$.

Let $\Hdim$ be a homological invariant of $R$-modules. We call $\Hdim$ \emph{closed under direct summands} provided that the following condition holds: whenever $X$ and $Y$ are $R$-modules, where $X$ is a direct summand of $Y$ and $\Hdim(Y)<\infty$, we have that $\Hdim(X)<\infty$. Next is the main result of this section:

\begin{thm} \label{t2} Let $\Hdim$ be a homological invariant of modules, and let $M$ be an $R$-module such that $\Ext_R^{i}(M,M)=0$ for all $i\geq 1$. Assume $\rHdim(M)<\infty$ with a reducing $\Ho$-sequence $\{K_0,  \ldots, K_r\}$. Then $M$ is a direct summand of $K_{i}$, and $\Ext_R^j(K_i, M)=0$ for each $i=0, \ldots, r$ and for all $j\geq 1$. In particular, if $\Hdim$ is closed under direct summands, then $\Hdim(M)<\infty$.
\end{thm}

\begin{proof} It is straightforward to derive the required conclusion of the second claim provided that the first one is correct: $M$ is a direct summand of $K_r$  so that $\Hdim(M)< \infty$ since $\Hdim(K_r)<\infty$. Hence we will prove the first claim by induction on $i$. In order to prove $M$ is a direct summand of $K_{i}$, we will show that there exists an $R$-module $L_i$ such that $K_i \cong M \oplus L_i$ for each $i$.

If $i=0$, then the claim follows by setting $L_i=0$ (recall $K_0=M$). Hence we assume $i\geq 1$. Then, by the induction hypothesis, $\Ext^{j}_R(K_{i-1},M)=0$ for all $j\geq 1$, and there exists an $R$-module $L_{i-1}$ such that $K_{i-1} \cong M \oplus L_{i-1}$. So, by applying $\Hom_R(-,M)$ to the exact sequence $0 \to K_{i-1}^{\oplus a_i} \overset{\varphi}{\to} K_i \to \Omega^{n_i}K_{i-1}^{\oplus b_i} \to 0$, we see that $\Ext^{j}_R(K_i,M)=0$ for all $j\geq 1$. Now we consider a pushout diagram of the map $\varphi$ and the split epimorphism $K_{i-1}^{\oplus a_i} \twoheadrightarrow L_{i-1}^{\oplus a_i}$.

$$\xymatrix{
& 0 \ar[d] & 0 \ar[d] \\
& M^{\oplus a_i}\ar[d] \ar@{=}[r] & M^{\oplus a_i} \ar[d] \\
 0 \ar[r] & K_{i-1}^{\oplus a_i} 
\ar[d]^{} \ar[r]^{} & K_i
\ar[d]^{} \ar[r]^{} &
\Omega^{n_i}K_{i-1}^{\oplus b_i}\ar@{=}[d] \ar[r] & 0 \\
0 \ar[r] & L_{i-1}^{\oplus a_i} \ar[r] \ar[d] & L_i' \ar[r] \ar[d] &
\Omega^{n_i}K_{i-1}^{\oplus b_i} \ar[r] & 0 \\
& 0 & 0 }$$
\vspace{0.01in}

As $K_{i-1} \cong M \oplus L_{i-1}$ and $\Ext^j_R(K_{i-1},M)=0$ for all $j\geq 1$, we conclude that $\Ext^1_R(L_{i-1},M)=0$.
Furthermore, since $\Ext^1(\Omega^{n_i}K_{i-1}^{\oplus b_i},M) \cong \Ext^{n_i+1}(K_{i-1},M)^{\oplus b_i}=0$, the bottom horizontal short exact sequence in the above diagram implies $\Ext^1_R(L_i',M)=0$.
This also implies that the middle vertical short exact sequence of the above diagram splits. So, $K_i \cong M^{\oplus a_i} \oplus L'_i$, and hence, by setting $L_i=M^{\oplus a_i-1} \oplus L_i'$, we see $K_i \cong M \oplus L_i$. This completes the proof.
\end{proof}

We can now establish the generalization we seek concerning the canonical module:

\begin{cor} \label{面白い} Let $C$ be a semidualizing $R$-module. Then the following conditions are equivalent:
\begin{enumerate}[\rm(i)]
\item $C \cong R$.
\item $C$ is totally reflexive.
\item $\G(C)<\infty$.
\item $\rGdim(C)<\infty$.
\end{enumerate}
In particular, if $R$ is Cohen-Macaulay with canonical module $\omega$, then $R$ is Gorenstein if and only if $\rGdim(\omega)<\infty$ if and only if $\G(\omega)<\infty$.
\end{cor}

\begin{proof} It suffices to prove (iv) implies (i). Hence we assume $\rGdim(C)<\infty$.  As $\G$ is closed under direct summands \cite[1.1.10(c)]{Chr} and since $\Ext^i_R(C,C)=0$ for all $i\geq 1$, Theorem \ref{t2} implies that $\G(C)<\infty$. In particular, $C$ is a reflexive complex; see \cite[2.2.3]{Chr} or \cite[2.7]{Yassemi}. Therefore \cite[5.3]{ATY2005} shows that $C\cong R$, and this establishes the assertion.
 
The second claim follows from \cite[1.2]{JLG} since $\omega$ is a semidualizing module.
\end{proof}

In passing, it is worth recalling that the conclusion of Corollary \ref{面白い} is not true if the canonical module is replaced by the residue field of the ring. In other words, in general, one does not have a characterization of regularity of local rings in terms of the reducing projective, or reducing Gorenstein dimension of the residue fields; see Example \ref{ex1}. 
\section{Proof of the main result}

The aim of this section is to prove our main result, namely Theorem \ref{mainthmintro}. Our proof relies upon the following proposition; its proof is deferred and is given in Section 5.

\begin{prop} \label{prop2} Let $M$ be an $R$-module such that $\rGdim(M)<\infty$. Then,
\begin{enumerate}[\rm(i)]
\item $\rGdim(\Omega M)<\infty$.\\
Assume further $\Ext_{R}^i(M,R)=0$ for all $i\geq 1$. Then the following hold:
\item If $M \cong \Omega N$ for some $R$-module $N$, then $\rGdim(N)<\infty$ 
\item If $K$ is an element of a reducing $\G$-sequence of $M$, then $\Ext_{R}^i(K,R)=0$ for all $i\geq 1$.
\qedhere
\popQED	
\end{enumerate}
\end{prop}

In order to prove Theorem \ref{mainthmintro}, we first recall:

\begin{rmk} \label{PF} Let $X$ be a torsionless $R$-module, i.e., $\lambda_{X}$ is injective, where $\lambda_{X}$ denotes the natural map $X \longrightarrow X^{\ast\ast}$. Consider a minimal free cover $F \stackrel{\pi} {\twoheadrightarrow} X^{\ast}$. By dualizing the map $\pi$, we obtain a short exact sequence $0 \to X \to F \to X_1 \to 0$, where $X_1$ is the cokernel of the composite map $X \stackrel{\lambda_X} {\longrightarrow} X^{\ast\ast} \stackrel{\pi^{\ast}} {\longrightarrow} F^{\ast} \cong F$. The short exact sequence $0 \to X \to F \to X_1 \to 0$ obtained in this way is called a \emph{pushforward} of $X$. Note the pushforward of $X$ is, up to isomorphism, unique and it follows, by construction, that $\Ext^1_R(X_1,R)=0$; see, for example, \cite[page 61, exercise 3(c)]{EG}. \pushQED{\qed} 
\qedhere
\popQED	
\end{rmk}

We are now ready to prove Theorem \ref{mainthmintro}:

\begin{proof}[Proof of Theorem \ref{mainthmintro}] We assume $\Ext^{i}_R(M,R)=0$ for all $i\gg 0$. To establish the theorem, it suffices to prove $\G(M)<\infty$; see \cite[23(c)]{Vla}. Note that, in view of Proposition \ref{prop2}(i), we may replace $M$ with a high syzygy module, assume $\Ext^{i}_R(M,R)=0$ for all $i\geq 1$, and proceed to prove $M$ is totally reflexive. For that, first, we will first show that $M$ is torsionless.

Set $r=\rGdim(M)$, and proceed by induction on $r$. If $r=0$, then, by Definition \ref{maindfn}, we have that $\G(M)<\infty$, and so $M$ is totally reflexive as $\Ext^{i}_R(M,R)=0$ for all $i\geq 1$. Thus we assume $r\geq 1$. Let $\{ K_0,K_1,\dots,K_r \}$ be a reducing $\G$ sequence of $M$. It then follows that $\rGdim(K_1)=r-1$ and $\Ext^i_R(K_1,R)=0$ for all $i\geq 1$; see Proposition \ref{prop2}(iii). So $K_1$ is torsionless by the induction hypothesis. 

It follows from Definition \ref{maindfn} that there is an exact sequence $0 \to M^{\oplus a} \to K_1 \to \Omega^nM^{\oplus b} \to 0$ for some positive integers $a$, $b$ and $n$. This implies that the dual sequence $0 \to (\Omega^nM^{\oplus b})^{\ast} \to K_1^{\ast} \to (M^{\oplus a})^{\ast} \to 0$ is also exact because $\Ext^1_R(\Omega^nM,R)=0$. Dualizing one more time, we obtain the following commutative diagram with exact rows, where $\lambda$ denotes the natural map:
\begin{center}
$\xymatrix{0 \ar[r] & M^{\oplus a} \ar[r] \ar[d]^{\lambda_{M^{\oplus a}}} & K_1 \ar[r] \ar[d]^{\lambda_{K_1}} & \Omega^nM^{\oplus b} \ar[r] \ar[d]^{\lambda_{\Omega^nM^{\oplus b}}} & 0\\
0 \ar[r] & (M^{\oplus a})^{\ast\ast} \ar[r] & K_1^{\ast\ast} \ar[r] & (\Omega^nM^{\oplus b})^{\ast\ast}}$
\end{center}
As $\lambda_{K_1}$ is injective, we see by the Snake lemma that $\lambda_{M^{\oplus a}}$ is injective. Since $\lambda_{M^{\oplus a}}=(\lambda_M)^{\oplus a}$, we conclude that the map $\lambda_M$ is injective, i.e., $M$ is torsionless, as claimed.

Notice, what we have proved above is that, if $X$ is an $R$-module with $\Ext^i_R(X,R)=0$ for all $i\geq 1$ and $\rGdim(X)<\infty$, then $X$ is torsionless.

Next we consider the pushforward of $M$, i.e., an exact sequence of $R$-modules $0 \to M \to F^{1} \to M_1 \to 0$, where $F^{1}$ is free; see Remark \ref{PF}. It then follows $\Ext^{i}_R(M_1,R)=0$ for all $i\geq 1$, and $\rGdim(M_1)<\infty$; see Proposition \ref{prop2}(ii). Thus, by what we have established above, we deduce that $M_1$ is torsionless. Moreover, the dual sequence $0 \to (M_1)^{\ast} \to (F^{1})^{\ast} \to M^{\ast} \to 0$ is exact. 

As $M_1$ is torsionless, we can iterate the previous process, and in this way we obtain exact sequences of $R$-modules $0 \to M_{i-1} \to F^{i} \to M_i \to 0$, where $F^{i}$ is free and $\Ext^{j}_R(M_i,R)=0$ for each $i, j\geq 1$. In particular, the dual sequence $0 \to (M_i)^{\ast} \to (F^i)^{\ast} \to (M_{i-1})^{\ast} \to 0$ is exact for each $i\geq 1$ (Set $M_{0}=M$). 
This gives us the following long exact sequence whose dual is also exact:

$$
\xymatrix{
0 \ar[r] & M \ar[r] & F^1 \ar[rr]\ar[rd] & & F^2 \ar[rr]\ar[rd] & & F^3 \ar[r] & \cdots \\
 &  & & M_1 \ar[ru]\ar[rd] & & M_2 \ar[ru]\ar[rd] & \\
 & & 0 \ar[ru] & & 0 \ar[ru] & & 0}
$$
\vspace{0.1in}

Now let $\cdots \to F_2 \to F_1 \to F_0 \to M \to 0$ be a free resolution of $M$. Splicing this free resolution with the one above, i.e., with the exact sequence $0 \to M \to F^1 \to F^2 \to F^3  \to \cdots$, we obtain a complete resolution of $M$ as follows:
$$
\xymatrix{
\cdots \ar[r] & F_1 \ar[r] & F_0\ar[rr]\ar[rd] & & F^1 \ar[r]  & F^2 \ar[r] & \cdots \\
 &  & & M \ar[ru]\ar[rd] & &   & \\
 & & 0 \ar[ru] & & 0  & & }
$$
\vspace{0.01in}

\noindent Therefore $M$ is totally reflexive and hence this completes the proof of the theorem.
\end{proof}

\section{Proof of Proposition \ref{prop2}}

This section is devoted to a proof of Proposition \ref{prop2}. Our proof requires several steps. Let us note here that Proposition \ref{prop2}(iii) follows from \ref{lemPMR}, while \ref{t1Main} establishes parts (i) and (ii) of the proposition.

\begin{chunk} \label{lemPMR}
Let $M$ and $N$ be $R$-modules such that $\rHdim(M)<\infty$ and $\PPP(M,N)=\sup\{i\in \ZZ:\Ext_R^i(M,N)\not=0\}<\infty$. If $\{K_0, K_1, \cdots, K_r\}$ is a reducing $\Hdim$-sequence of $M$, then $\PPP(K_i, N)=\PPP(M,N)$  for each $i=0, \ldots, r$.
\end{chunk}

\begin{proof} Let $\{K_0, K_1, \cdots, K_r\}$ be a reducing $\Hdim$-sequence of $M$. It follows from Definition \ref{maindfn} that, given an integer $i$ with $1\leq i \leq r$, there exists a short exact sequence of $R$-modules:
\begin{equation} \tag{\ref{lemPMR}.1}
0 \to K_{i-1}^{{\oplus a_i}} \to K_{i} \to \Omega^{n_i}K_{i-1}^{{\oplus b_i}} \to 0,
\end{equation}
where $a_1, \dots, a_r,b_1, \dots, b_r ,n_1, \dots, n_r$ are positive integers, $K_0=M$ and $\Hdim(K_r)<\infty$.

We will first observe, for each $i=1, \ldots, r$, that $\PPP(K_i, N)<\infty$, and then $\PPP(K_{i-1}, N)=\PPP(K_i, N)$. To establish both of these claims, we will proceed by induction on $i$.

If $i=0$, then $\PPP(K_i, N)=\PPP(M,N)$, which is finite by assumption. So we assume $i\geq 1$. Since $\Ext_R^j(\Omega^{n_i}K_{i-1},N) \cong \Ext_R^{j+n_i}(K_{i-1},N)$ for each $j\geq 1$, and since the induction hypothesis gives $\PPP(K_{i-1}, N)<\infty$, we conclude that $\Ext_R^{v}(K_{i-1},N)=0$ for all $v\gg 0$, i.e., $\PPP(K_i, N)<\infty$. Hence it remains to show $\PPP(K_{i-1}, N)=\PPP(K_i, N)$ for each $i=1, \ldots, r$. 

We pick an integer $i$ with $1\leq i\leq r$, set $p=\PPP(K_{i-1},N)$, and consider the following long exact sequence which follows from (\ref{lemPMR}.1): 
\begin{equation} \tag{\ref{lemPMR}.2}
\cdots \to \Ext_R^{j}(\Omega^{n_i} K_{i-1}^{{\oplus b_i}}, N) \to \Ext_R^{j}(K_i,N) \to \Ext_R^{j}(K_{i-1}^{{\oplus a_i}},N) \to \Ext_R^{j+1}(\Omega^{n_i} K_{i-1}^{{\oplus b_i}}, N) \to \cdots
\end{equation}
Letting $j=p$ in (\ref{lemPMR}.2), we see that $\Ext_R^{p}(K_{i},N) \cong \Ext_R^{p}(K_{i-1}^{{\oplus a_i}},N)\neq 0$. Furthermore, if $j\geq p+1$, then $\Ext_R^{j}(\Omega^{n_i} K_{i-1}^{{\oplus b_i}}, N)=0=\Ext_R^{j}(K_{i-1}^{{\oplus a_i}},N)$ so that $\Ext_R^{j}(K_i,N) =0$.
This shows $\PPP(K_{i-1},N)=\PPP(K_i,N)$. Therefore we conclude that $\PPP(M,N)=\PPP(K_{0},N)=\cdots=\PPP(K_r,N)<\infty$. This completes the proof.
\end{proof}

In our proof of \ref{t1Main}, we will make use of the next result; it is an application of the Horseshoe Lemma and hence we skip its proof.

\begin{chunk} \label{HSL} Let $M$, $N$ and $L$ are $R$-modules.
\begin{enumerate}[\rm(i)]
\item If $\theta=(0 \to N \to L \to M \to 0)$ is an exact sequence, then there exists a free $R$-module $F$ such that the syzygy sequence $\Omega \theta=(0 \to \Omega N \to \Omega L \oplus F \to \Omega M \to 0)$ is exact.
\item Let $f_{M}^{N}:\Ext^1_R(M,N) \to \Ext^1_R(\Omega M, \Omega N)$ be the map given by $f_{M}^{N}(\theta)=\Omega \theta$, where $\Omega \theta$ is defined as in part (i). If $\Ext_R^2(M,R)=0$, then $f_{M}^{N}$ is surjective. 
\pushQED{\qed} 
\qedhere
\popQED	
\end{enumerate}
\end{chunk}

\begin{chunk} \label{t1Main} Let $M$ and $N$ be $R$-modules such that $M \cong \Omega N$. Assume $\Hdim$ satisfies the following properties:
\begin{enumerate}[\rm (1)]
\item $\Hdim(R)<\infty$.
\item $\Hdim(X)<\infty \iff \Hdim(\Omega X)<\infty$ for each $R$-module $X$.
\item $\Hdim(X)<\infty \iff \Hdim(X\oplus R)<\infty$ for each $R$-module $X$. 
\end{enumerate}
Then the following hold:
\begin{enumerate}[\rm(i)]
\item If $\rHdim(N)<\infty$, then $\rHdim(M)<\infty$.
\item If $\rHdim(M)<\infty$ and $\Ext_R^{i}(M,R)=0$ for all $i\geq 1$, then $\rHdim(N)<\infty$.
\end{enumerate}
In particular, the result holds for the case where $\Hdim=\G$, i.e., if $\rGdim(N)<\infty$, then $\rGdim(M)<\infty$, and if $\rGdim(M)<\infty$ and $\Ext_R^{i}(M,R)=0$ for all $i\geq 1$, then $\rGdim(N)<\infty$.
\end{chunk}

\begin{proof}
(i) Assume $\rHdim(N)<\infty$ and let $\{ K_0,\dots,K_r\}$ be a reducing $\Hdim$-sequence of $N$. Then there are positive integers $a_1,\dots,a_r,b_1,\dots,b_r,n_1,\dots,n_r$, and exact sequences
\begin{equation} \tag{\ref{t1Main}.1}
0 \to K_{i-1}^{\oplus a_i} \to K_i \to \Omega^{n_i}K_{i-1}^{\oplus b_i} \to 0,
\end{equation}
where $1\leq i \leq r$, $K_{0}=N$ and $\Hdim(K_r)<\infty$; see Definition \ref{maindfn}. We will prove that there is a reducing $\Hdim$-sequence $\{Z_0,\dots, Z_r\}$ of $M$, where $Z_i \cong \Omega K_i \oplus F_i$ for some free $R$-module $F_i$ and for each $i=0, \ldots r$.

We set $Z_0=M$ and $F_0=0$ so that, for the case where $i=0$, we have $Z_0 \cong \Omega K_0 \oplus F_0 = \Omega N \oplus 0 \cong M$ by assumption. Hence suppose $i\geq 1$. Then, by the induction hypothesis, there are $R$-modules $Z_{i-1}$ and $F_{i-1}$ with $F_{i-1}$ is free and $Z_{i-1} \cong \Omega K_{i-1} \oplus F_{i-1}$.

We use \ref{HSL}(i) with the sequence (\ref{t1Main}.1), and obtain free $R$-modules $G_i$ and the exact sequence:
\begin{equation} \tag{\ref{t1Main}.2}
0 \to \Omega K_{i-1}^{\oplus a_i} \overset{\varphi}{\longrightarrow} \Omega K_i\oplus G_i \to \Omega^{n_i+1}K_{i-1}^{\oplus b_i} \to 0.
\end{equation}

Next we take the direct sum of the sequence (\ref{t1Main}.2) with the trivial sequence $0 \to F_{i-1}^{\oplus a_i} {\to} F_{i-1}^{\oplus a_i} \to 0$, and obtain the exact sequence:
\begin{equation} \tag{\ref{t1Main}.3}
0 \to \Omega K_{i-1}^{\oplus a_i} \oplus F_{i-1}^{\oplus a_i} \overset{\left(\begin{smallmatrix}
\varphi \\ 1
\end{smallmatrix}\right)}{\longrightarrow} (\Omega K_i \oplus G_i) \oplus F_{i-1}^{\oplus a_i} \to \Omega^{n_i+1}K_{i-1}^{\oplus b_i} \to 0.
\end{equation}

As $n_i$ is positive, the exact sequence (\ref{t1Main}.3) can be written as follows:
\begin{equation} \tag{\ref{t1Main}.4}
0 \to \big(\Omega K_{i-1} \oplus F_{i-1}\big)^{\oplus a_i}         
\overset{\left(\begin{smallmatrix} \varphi \\ 1\end{smallmatrix}\right)}{\longrightarrow} 
\Omega K_i \oplus \big(G_i \oplus F_{i-1}^{\oplus a_i}\big) \to\Omega^{n_i}\big(\Omega K_{i-1} \oplus F_{i-1}\big)^{\oplus b_i}  \to 0.
\end{equation}

Now we set $F_i=G_i \oplus F_{i-1}^{\oplus a_i}$ and $Z_i=\Omega K_i \oplus F_i$. Then, by making use of (\ref{t1Main}.4), for each $i=1, \ldots, r$, we obtain the following exact sequences:
\begin{equation} \tag{\ref{t1Main}.5}
0 \to Z_{i-1}^{\oplus a_i} \to Z_i \to \Omega^{n_i}Z_{i-1}^{\oplus b_i} \to 0.
\end{equation}

Note that $\Hdim(K_r)<\infty \Longrightarrow \Hdim(\Omega K_r)<\infty \Longrightarrow \Hdim(\Omega K_r \oplus F_r)<\infty \Longrightarrow \Hdim(Z_r)<\infty$: the first implication follows from (2), while the second one is due to (3) of the hypotheses. As $Z_0=M$, (\ref{t1Main}.5) shows that $\{Z_0,\dots, Z_r\}$ is an $\Hdim$-sequence of $M$. So $\rHdim(M)<\infty$, as required. This justifies part (i).

(ii) Assume $\rHdim(M)<\infty$ and $\Ext_R^{i}(M,R)=0$ for all $i\geq 1$. Let $\{ K_0,\dots,K_r\}$ be a reducing $\Hdim$-sequence of $M$. Then there are positive integers $a_1,\dots,a_r,b_1,\dots,b_r,n_1,\dots,n_r$, and exact sequences
\begin{equation} \tag{\ref{t1Main}.6}
0 \to K_{i-1}^{\oplus a_i} \to K_i \to \Omega^{n_i}K_{i-1}^{\oplus b_i} \to 0,
\end{equation}
where $1\leq i \leq r$, $K_{0}=M$ and $\Hdim(K_r)<\infty$; see Definition \ref{maindfn}. Note that, $\Ext^j_R(K_i,R)=0$ for all $j\geq 1$ and for all $i=0, \ldots r$; see \ref{lemPMR}.

We will prove that there is a reducing $\Hdim$-sequence $\{W_0,\dots, W_r\}$ of $N$, where $K_i \cong \Omega W_i \oplus F_i$ for some free $R$-module $F_i$ and for each $i=1, \ldots r$.

We set $W_0=N$ and $F_0=0$ so that, for the case where $i=0$, we have $K_0 \cong \Omega W_0 \oplus F_0 = \Omega N \oplus 0 \cong M$ by assumption. Hence suppose $i\geq 1$. Then, by the induction hypothesis, there are $R$-modules $W_{i-1}$ and $F_{i-1}$ with $F_{i-1}$ is free and $K_{i-1} \cong \Omega W_{i-1} \oplus F_{i-1}$. Note this isomorphism gives a split epimorphism $K_{i-1}^{\oplus a_i} \to \Omega W_{i-1}^{\oplus a_i}$. Next we consider a pushout of this split epimorphism with the map $K_{i-1}^{\oplus a_i} \to K_i$ that comes from (\ref{t1Main}.6):
$$\xymatrix{
& 0 \ar[d] & 0 \ar[d] \\
& F_{i-1}^{\oplus a_i}\ar[d] \ar@{=}[r] & F_{i-1}^{\oplus a_i} \ar[d] \\
 0 \ar[r] & K_{i-1}^{\oplus a_i} 
\ar[d]^{} \ar[r]^{} & K_i
\ar[d]^{} \ar[r]^{} &
\Omega^{n_i}K_{i-1}^{\oplus b_i} \ar@{=}[d] \ar[r] & 0 \\
0 \ar[r] & \Omega W_{i-1}^{\oplus a_i} \ar[r] \ar[d] & \exists \; L_i \ar[r] \ar[d] &
\Omega^{n_i}K_{i-1}^{\oplus b_i} \ar[r] & 0 \\
& 0 & 0 }$$
\vspace{0.01in}

We have $\Ext^1_R(\Omega W_{i-1}^{\oplus a_i},R)=0=\Ext^1_R(\Omega^{n_i}K_{i-1}^{\oplus b_i},R)$ since $\Ext^{j}(K_{i-1},R)=0$ for all $j\geq 1$ and since $\Omega W_{i-1}$ is a direct summand of $K_{i-1}$. Therefore, by the bottom horizontal short exact sequence above, we conclude $\Ext^1_R(L_i,R)=0$. This shows, since $F_{i-1}$ is free, that the middle vartical exact sequence splits. Consequently, we obtain the isomorphism $K_i \cong L_i \oplus F_{i-1}^{\oplus a_i}$.

As $\Ext^{j}(K_{i-1}, R)=0$ for all $j\geq 1$, we have $\Ext^2_R(\Omega^{n_i-1}K_{i-1}^{\oplus b_i},R)=0$. Therefore the map $\Ext^1_R(\Omega^{n_i-1}K_{i-1}^{\oplus b_i},W_{i-1}^{\oplus a_i}) \to \Ext^1_R(\Omega^{n_i}K_{i-1}^{\oplus b_i},\Omega W_{i-1}^{\oplus a_i})$, given by taking syzygy, is surjective; see \ref{HSL} (ii). In particular, there exists $\theta= \big( 0 \to W_{i-1}^{\oplus a_i} \to P_i \to 
\Omega^{n_i-1}K_{i-1}^{\oplus b_i} \to 0 \big) \in \Ext^1_R(\Omega^{n_i-1}K_{i-1}^{\oplus b_i},W_{i-1}^{\oplus a_i})$ whose syzygy is the short exact sequence $\big(0 \to \Omega W_{i-1}^{\oplus a_i}  \to L_i \to \Omega^{n_i}K_{i-1}^{\oplus b_i} \to 0 \big) \in \Ext^1_R(\Omega^{n_i}K_{i-1}^{\oplus b_i},\Omega W_{i-1}^{\oplus a_i})$. It follows, by the definition of the syzygy map, that there exists a free $R$-module $G_i$ with $L_i \cong \Omega P_i \oplus G_i$.

In passing, we summarize the isomorphisms we obtained so far:
\begin{equation} \tag{\ref{t1Main}.7}
K_{i-1} \cong \Omega W_{i-1} \oplus F_{i-1}  \text{, and } K_i \cong L_i \oplus F_{i-1}^{\oplus a_i} \text{, and } L_i \cong \Omega P_i \oplus G_i \text{, where } F_{i-1}, F_i \text{ and } G_i \text{ are free}.
\end{equation}
Now we consider two cases:\\
\emph{Case 1}: Assume $n_i=1$. In this case, notice $\theta= \big( 0 \to W_{i-1}^{\oplus a_i} \to P_i \to K_{i-1}^{\oplus b_i} \to 0 \big)$. As the isomorphism $K_{i-1} \cong \Omega W_{i-1} \oplus F_{i-1}$ yields
the isomorphism $K_{i-1}^{\oplus b_i} \cong \Omega W_{i-1}^{\oplus b_i} \oplus F_{i-1}^{\oplus b_i}$, we can consider the pullback of the split monomorphism $\Omega W_{i-1}^{\oplus b_i} \to K_{i-1}^{\oplus b_i}$ and the map
$P_i \to K_{i-1}^{\oplus b_i}$ that comes from $\theta$: 

$$\xymatrix{
 & & 0 \ar[d] & 0 \ar[d] \\
0 \ar[r]& W_{i-1}^{\oplus a_i} \ar@{=}[d] \ar[r] & \exists \; W_i \ar[d] \ar[r] & \Omega W_{i-1}^{\oplus b_i} \ar[d] \ar[r] & 0 \\
0 \ar[r]& W_{i-1}^{\oplus a_i} \ar[r] & P_i \ar[d] \ar[r]^{} & K_{i-1}^{\oplus b_i} \ar[d] \ar[r] & 0 \\
& & F_{i-1}^{\oplus b_i} \ar[d] \ar@{=}[r] & F_{i-1}^{\oplus b_i} \ar[d] \\
& & 0 & 0
}$$
\vspace{0.01in}

\noindent As $F_{i-1}$ is free, it follows that $P_i \cong W_i \oplus F_{i-1}^{\oplus b_i}$, and hence $\Omega P_i \cong \Omega W_i$.

\noindent \emph{Case 2}: Assume $n_i\geq 2$. Then ${n_i}-1>0$ and, since $K_{i-1} \cong \Omega W_{i-1} \oplus F_{i-1}$, it follows that:
\begin{equation} \tag{\ref{t1Main}.8}
\Omega^{n_i-1}K_{i-1}^{\oplus b_i} \cong \Omega^{n_i-1}(\Omega W_i \oplus F_{i-1})^{\oplus b_i}  \cong \Omega^{n_i}W_{i-1}^{\oplus b_i}. 
\end{equation}
In this case we define $W_i$ as $P_i$, i.e., we set $W_i=P_i$.

In both Case 1 and 2, we have the exact sequence $0 \to W_{i-1}^{\oplus a_i} \to W_i \to \Omega^{n_i} W_{i-1}^{\oplus b_i} \to 0$: for Case 1, this is obtained in the previous pullback diagram, and for Case 2, this follows from the definition of $\theta$, the fact that $W_i=P_i$, and the isomorphism (\ref{t1Main}.8). Furthermore, for both cases, we have the following isomorphism; see also (\ref{t1Main}.7).
\begin{equation} \tag{\ref{t1Main}.9}
K_i \cong L_i \oplus F_{i-1}^{\oplus a_i} \cong (\Omega P_i \oplus G_i) \oplus F_{i-1}^{\oplus a_i} \cong \Omega W_i  \oplus F_i \text{, where } F_i=G_i \oplus F_{i-1}^{\oplus a_i}.
\end{equation}
The isomorphism (\ref{t1Main}.9), in particular, implies that $\Hdim(K_r)=\Hdim(\Omega W_r \oplus F_r)<\infty$. Thus, from the hypotheses (2) and (3), we conclude that $\Hdim(W_r)<\infty$. Consequently,  
$\{W_0,\dots, W_r\}$ is an $\Hdim$-sequence of $N$, and hence it follows that $\Hdim(N)<\infty$. This completes the proof of the theorem.
\end{proof}

We finish this section by recording the following observation about Theorem \ref{t1Main}:

\begin{rmk} Let $R=k[\![x,y]\!]/(x,y)^2$ for some field $k$, and let $N=R/(x)$. Then $M = \Omega N = (x) \cong k$. Hence it follows from Proposition \ref{prop1} that $\rGdim(M)<\infty$ and $\rGdim(N)=\infty$
(since $M\cong k$ and since $N$ is not isomorphic to $R^{\oplus \alpha} \oplus k^{\oplus \beta}$ for all $\alpha, \beta \geq 0$). Note that $\Ext^i_R(M,R)\neq 0$ for each $i\geq 0$.
This shows, in general, the converse of Theorem \ref{t1Main}(i) is not true, and that the vanishing of $\Ext^i_R(M,R)$ assumption is necessary for Theorem \ref{t1Main}(ii) to hold.
\end{rmk}

\section*{Acknowledgements}
The authors are grateful to Hiroki Matsui and Yuji Yoshino for explaining Example \ref{withoutrmkEx2} to them. The authors are indebted to Mohsen Asgharzadeh and Hiroki Matsui for reading the manuscript and giving helpful comments and suggestions.


\begin{thebibliography}{10}

\bibitem{ATY2005}
Tokuji Araya, Ryo Takahashi, and Yuji Yoshino, \emph{Homological invariants
  associated to semi-dualizing bimodules}, J. Math. Kyoto Univ. \textbf{45}
  (2005), no.~2, 287--306.

\bibitem{AuBr}
Maurice Auslander and Mark Bridger, \emph{Stable module theory}, Memoirs of the
  American Mathematical Society, No. 94, American Mathematical Society,
  Providence, R.I., 1969.

\bibitem{Av2}
Luchezar~L. Avramov, \emph{Infinite free resolutions}, six lectures on
  commutative algebra ({B}ellaterra, 1996), Progr. Math., vol. 166,
  Birkh\"auser, Basel, 1998, pp.~1--118.

\bibitem{AGP}
Luchezar~L. Avramov, Vesselin~N. Gasharov, and Irena~V. Peeva., \emph{Complete
  intersection dimension}, Inst. Hautes \'Etudes Sci. Publ. Math. (1997),
  no.~86, 67--114.

\bibitem{Be}
Petter~Andreas. Bergh, \emph{Modules with reducible complexity}, J. Algebra
  \textbf{310} (2007), 132--147.

\bibitem{BH}
Winfried Bruns and J{{\"u}}rgen Herzog, \emph{Cohen-{M}acaulay rings},
  Cambridge Studies in Advanced Mathematics, vol.~39, Cambridge University
  Press, Cambridge, 1993.

\bibitem{Chr}
Lars~Winther Christensen, \emph{Gorenstein dimensions}, Lecture Notes in
  Mathematics, vol. 1747, Springer-Verlag, Berlin, 2000.

\bibitem{EG}
E.~Graham Evans and Phillip Griffith, \emph{Syzygies}, London Mathematical
  Society Lecture Note Series, vol. 106, Cambridge University Press, Cambridge,
  1985.

\bibitem{JLG}
David~A. Jorgensen and Graham~J. Leuschke, \emph{On the growth of the betti
  sequence of the canonical module}, Math. Z. \textbf{256} (2007), no.~3,
  647--659.

\bibitem{JS}
David~A. Jorgensen and Liana~M. {\c{S}}ega, \emph{Independence of the total
  reflexivity conditions for modules}, Algebras and Representation Theory
  \textbf{9} (2006), no.~2, 217--226.

\bibitem{BLW}
Graham~J. Leuschke and Roger Wiegand, \emph{Cohen-{M}acaulay representations},
  Mathematical Surveys and Monographs, vol. 181, American Mathematical Society,
  Providence, RI, 2012.

\bibitem{Vla}
Vladimir Ma{\c{s}}iek, \emph{Gorenstein dimension and torsion of modules over
  commutative {N}oetherian rings}, Communications in Algebra \textbf{28}
  (2000), no.~12, 5783--5811.

\bibitem{SSC}
Sean Sather-Wagstaff, \emph{Semidualizing modules and the divisor class group},
  Illinois J. Math. \textbf{51} (2007), no.~1, 255--285. \MR{2346197}

\bibitem{TakG}
Ryo Takahashi, \emph{On {G}-regular local rings}, Communications in Algebra
  \textbf{36} (2008), no.~12, 4472--4491.

\bibitem{Yassemi}
Siamak Yassemi, \emph{G-dimension}, Math. Scand. \textbf{77} (1995), 161--174.

\bibitem{Yo}
Yuji Yoshino, \emph{Cohen-{M}acaulay modules over {C}ohen-{M}acaulay rings},
  London Mathematical Society Lecture Note Series, vol. 146, Cambridge
  University Press, Cambridge, 1990.

\bibitem{YYGdim}
\bysame, \emph{Modules of {G}-dimension zero over local rings with the cube of
  maximal ideal being zero}, Commutative algebra, singularities and computer
  algebra, Springer, 2003, pp.~255--273.

\bibitem{Yuj}
\bysame, \emph{Homotopy categories of unbounded complexes of projective
  modules}, preprint; posted on the arXiv:1805.05705 (2019).

\end{thebibliography}
\end{document}